\documentclass[fleqn,12pt,twoside]{article}
\usepackage{a4wide}
\usepackage[active]{srcltx}
\usepackage{mathtools}
\usepackage[usenames,dvipsnames]{color}
\usepackage{graphicx,palatino,pifont,times}
\usepackage{lscape}
\usepackage{amsmath,amsthm}
\usepackage{amssymb}
\usepackage{amstext}
\usepackage{epsfig}
\usepackage{multirow}
\usepackage{enumerate}
\usepackage{cite}
\usepackage[active]{srcltx}
\catcode`\@=11
\def\theequation{\@arabic{\c@section}.\@arabic{\c@equation}}
\catcode`\@=12

\newtheorem{theorem}{Theorem}[section]
\newtheorem{definition}{Definition}[section]
\newtheorem{remark}{Remark}[section]
\newtheorem{proposition}{Proposition}[section]
\newtheorem{corollary}{Corollary}[section]
\newtheorem{lemma}{Lemma}[section]

\newtheorem{example}{Example}[section]
\catcode`\@=11
\def\theequation{\@arabic{\c@section}.\@arabic{\c@equation}}
\catcode`\@=12
\theoremstyle{remark}
\newtheorem{case}{Case}

\usepackage{enumitem}

\title{Box dimension of fractal functions on attractors}

\begin{document}
	\date{}
	\maketitle
	
	

	\begin{center}
	 R. Pasupathi\\
Sobolev Institute of Mathematics SB RAS, 630090, Novosibirsk, Russia \\
	 Email: pasupathi4074@gmail.com
	\end{center}

	\begin{abstract}
	We study a wide class of fractal interpolation functions in a single platform by considering the  domains of these functions as general attractors. We obtain lower and upper bounds of the box dimension of these functions in a more general setup where the interpolation points need not be equally spaced, the scale vectors can be variables and the maps in the corresponding IFS can be non-affine. In particular, we obtain the exact value of the box dimension of non-affine fractal functions on general m-dimensional cubes and Sierpi\'nski Gasket.
	\end{abstract}
	
	
\section{Introduction}\label{sec1}
In 1986, Barnsley \cite{b1} proposed the concept of a fractal interpolation function (FIF) on intervals using the notion of iterated function system. This work was extended to many different domains such as triangles \cite{Ma1}, $m$-dimensional cubes \cite{fis-13}, Sierpi\'nski Gasket \cite{DSO}, post critically finite (p.c.f.) self-similar sets  \cite{rua} etc. 

The graph of a function, and its box and Hausdorff dimensions, have been of qualitative interest for many researchers since the past few decades. Several theories concerning the box dimension of fractal functions have been explored in the literature. Some of recent works on FIF and dimension theory can be found in \cite{cvvt,KUM,jcns,nava,PRS}.

Hardin and Massopust \cite{di1} have estimated the value of the box dimension of graph of FIF on an interval when the maps in the corresponding IFS are affine and the interpolation points are equally spaced. Barnsley and Massopust \cite{bm} studied the box dimension of bilinear FIF on an interval in the case of equally spaced data points. Nasim et al. \cite{agn} obtained the box dimension of  non-affine FIF  in the case of equally spaced data points  by using H\"{o}lder exponent. Feng and Sun \cite{fesu} studied the box dimension of FIFs on a rectangle derived from affine FIFs on an interval with arbitrary interpolation nodes. Geronimo and Hardin \cite{GEHA} have estimated the box dimension of self-affine FIF  on   
polygonal regions. Bouboulis et al. \cite{bdd}
introduced recurrent FIF (RFIF) on a rectangle as the invariant set of the recurrent IFS (RIFS) in order to gain more flexibility, and studied the  value of the box dimension of RFIF when the maps in the corresponding RIFS are affine with respect to each variable and the interpolation points are equally spaced. Bouboulis and Dalla \cite{bd} generalized the theory of RFIF and its box dimension to higher dimensions.

In the literature, we found that, in most of the cases, the authors assumed that the interpolation points are equally spaced and the maps in the corresponding IFS are affine when the box dimension of  FIF was considered, and  there  was no discussion about the box dimension of FIFs on  p.c.f. self-similar sets except the Sierpi\'{n}ski Gasket. Even in the case of Sierpi\'{n}ski Gasket also, the authors \cite{bdim} obtained only a non-trivial  upper bound of the box dimension of FIFs by using H\"{o}lder exponent. The authors always assumed all the scale variables to be constants for estimating a non-trivial lower bound of the box dimension of fractal functions.

In this paper, we consider FIFs on attractors  and we study  lower and upper bounds of the box dimension of these functions without assuming that the interpolation points are equally spaced, the maps in the corresponding IFSs are affine and all the scale vectors are constants. We derive upper bounds of the box dimension of fractal functions using a function space called the oscillation space, which contains the collection of H\"{o}lder continuous functions. We provide non-trivial lower bounds of the box dimension of FIF on $m$-dimensional cubes and Sierpi\'{n}ski Gasket.
\section{Preliminaries}
\begin{definition}
	An \textit{iterated function system (IFS)} consists of a complete metric space
	$(X,d)$ together with a finite set of continuous mappings $f_i : X \to X, i = 1, 2, \dots , N$. We denote it as $\mathcal{S}=\{(X,d), (f_i)_{i=1}^N\}$.
\end{definition}
For an IFS, we can define  an operator (called Hutchinson operator) $\mathcal{F}_\mathcal{S}:(\mathcal{H}(X),h)\to (\mathcal{H}(X),h)$ by
$$\mathcal{F}_\mathcal{S}(B)=\underset{i=1}{\overset{N}{\cup}}f_i(B)\;\;\; \text{for}\; B\in \mathcal{H}(X),$$
where $(\mathcal{H}(X),h)$ is the  Hausdorff metric space on $X$ i.e., $\mathcal{H}(X)$ is the collection of all non-empty compact subsets of $X$ and $h$ is the  Hausdorff distance.\\
If $\mathcal{F}_\mathcal{S}$ has a unique fixed point $A_\mathcal{S}$ (say) and 	$\underset{n \to \infty}{\lim}\mathcal{F}_\mathcal{S}^{[n]}(B) = A_\mathcal{S}$ for every $B\in \mathcal{H}(X)$, then $A_\mathcal{S}$ is called the \textit{attractor}  of the IFS, where by $f^{[n]}$, we mean the composition of a function $f$ with itself $n$ times.
\begin{definition}
	Let $(X, d)$ be  a metric space and $f : X \rightarrow X$ be a function. If there is a constant $ r \in [0,1)$ such that:
	\begin{itemize}
		\item[$-$]  
		$$	d(f(x), f(y)) \leq r ~ d(x,y)~~\text{for}~ x, y \in X,$$
		then $f$ is called contraction;
		\item[$-$] 	$$	d(f(x), f(y)) = r ~ d(x,y)~~\text{for}~ x, y \in X,$$ 	then $f$ is called similarity.
	\end{itemize}
	The constant $r$ is called the contractivity factor of $f$.
\end{definition}
\begin{remark}\cite{HUT}\label{1thm1}
	An IFS $\{(X,d), (f_i)_{i=1}^N\}$ has an attractor 
	provided that $f_i$'s are contractions.
\end{remark}
\begin{remark}
	If a function $f:\mathbb{R}^k\to \mathbb{R}^k$ is similarity, then it is an affine map.
\end{remark}
For an IFS $\mathcal{S}=\{(X,d), (f_i)_{i=1}^N\}$, we denote $\mathcal{N}=\{1,2,\dots,N\}$, $\mathcal{N}^k=\left\{(\omega_1,\omega_2,\dots ,\omega_k):\omega_i\in \mathcal{N}\right\}$ and $\mathcal{N}^{\infty}=\left\{(\omega_i)_{i=1}^{\infty} :\omega_i\in \mathcal{N}\right\}$, and we define $$f_{\omega}=f_{\omega_1}\circ f_{\omega_2}\circ\dots \circ f_{\omega_k}\;\;\;\text{for}\;\omega=(\omega_1,\omega_2,\dots ,\omega_k)\in \mathcal{N}^k,k\in \mathbb{N}.$$
Let $\pi:\mathcal{N}^{\infty}\to A_\mathcal{S}$ be defined by $$\pi(\omega)=\underset{k\in \mathbb{N}}{\bigcap}f_{\omega_1\omega_2\cdots \omega_k}(A_\mathcal{S})~~~~\text{for}~~\omega \in \mathcal{N}^{\infty}.$$
Following Kigami \cite{KI}, we define the \textit{critical set} $\mathcal{C}$ and the \textit{post critical set}  $\mathcal{P}$ of $A_\mathcal{S}$ by
$$\mathcal{C}=\pi^{-1}\left(\underset{\substack{i,j\in \mathcal{N}\\i\neq j}}{\bigcup}\left(f_i(A_\mathcal{S})\cap f_j(A_\mathcal{S})\right)\right)~~~\text{and}~~~\mathcal{P}=\underset{k\in \mathbb{N}}{\bigcup}\tau^{[k]}(\mathcal{C}),$$
where $\tau$ is the left shift operator on $\mathcal{N}^{\infty}$.  If $\mathcal{P}$ is a finite set, then 
$A_\mathcal{S}$ is called a \textit{post critical finite (p.c.f.) self-similar set}. The \textit{boundary} of $A_\mathcal{S}$ is defined by $V_0^*=\pi(\mathcal{P})$ and we define
$$V_k=\underset{\omega \in \mathcal{N}^k}{\bigcup}f_{\omega}(V_0^*)~~~\text{for}~k\in \mathbb{N}.$$
\begin{remark}\label{131re1}\cite{KI}
	\begin{itemize}
		\item[(i)] $V_0^*\subseteq V_1\subseteq V_2\subseteq \cdots V_{k-1} \subseteq V_k\subseteq \cdots$.
		\item[(ii)] 	For $\omega,\omega' \in \mathcal{N}^k,k\in \mathbb{N}$ with $\omega \neq \omega'$, we get
		$$~~~~f_{\omega}(A_\mathcal{S})\cap f_{\omega'}(A_\mathcal{S})=f_{\omega}(V_0^*)\cap f_{\omega'}(V_0^*). $$
	\end{itemize}
	
\end{remark}	
\begin{definition}
	For a nonempty bounded subset $F$ of $\mathbb{R}^k$, the lower and upper box dimensions are defined by
	$$\underline{\dim}_B(F)=\liminf\limits_{\delta \to 0^+}\frac{\log {N}_{\delta}(F) }{\log \left(\frac{1}{\delta}\right)}~~~~~\text{and}~~~~~~\overline{\dim}_B(F)=\limsup\limits_{\delta \to 0^+}\frac{\log {N}_{\delta}(F) }{\log \left(\frac{1}{\delta}\right)},$$
	where	${N}_{\delta}(F)$ is the  minimum number of boxes with side length $\delta$ and sides parallel to the axes, whose	union contains  $F$.
	If $\underline{\dim}_B(F)=\overline{\dim}_B(F)$, this common value is denoted by $\dim_B(F)$ and is called the  \textit{box dimension} or \textit{Minkowski dimension} of $F$.
\end{definition}
\begin{remark}\cite{FAL}
	For a nonempty bounded 	subset $F$ of $\mathbb{R}^k$ and a continuous function $f : F \to \mathbb{R}$, we have  the following inequalities
	$$\dim_H(F)\leq \underline{\dim}_B(F)\leq \overline{\dim}_B(F)\;\;\text{and}\;\;\dim_H(F)\leq \dim_H(G_f),$$
	where $\dim_H(F)$ means the Hausdorff dimension of $F$.
\end{remark}
\begin{definition}
	Let $(X, d_X)$ and $(Y,d_Y)$ be  metric spaces. A function $f:X\to Y$ is called  H\"{o}lder continuous with exponent $\eta$ if  $\eta\in (0,1]$ and there exists $H\in [0,\infty)$ such that
	$$d_Y(f(x),f(x'))\leq H ~d_X(x,x')^{\eta}~~\text{for}~x,x'\in X.$$
\end{definition}
\section{Fractal interpolation function}\label{2401s3}
Let $p:V:=\underset{i\in \mathcal{N}}{\bigcup}l_i(V_0)\to \mathbb{R}$ be a given  function (data), where   
\begin{itemize}
	\item[$-$]  $V_0=\{k_1,k_2,\dots,k_r\}\subseteq V\subseteq K$,
	\item[$-$] $K$ is an attractor of an IFS  $\{(\mathbb{R}^m,\|.\|_2),(l_i)_{i\in \mathcal{N}}\}$, $\|.\|_2$ is the Euclidean metric on $\mathbb{R}^m$, $\mathcal{N}=\{1,2,\dots,N\}$ and  $l_i$ is a similarity map on $\mathbb{R}^m$ with the contractivity factor $r_i$ for $i\in \mathcal{N}$.
\end{itemize}  
Let us consider $g_i:K \times \mathbb{R}\to \mathbb{R}$ defined by
\begin{equation}\label{0702e1}
	g_i(x,z)=s_{i}(x)z+q_{i}(x)\;\;\;\text{for}\; (x,z)\in K \times \mathbb{R}, i\in \mathcal{N},
\end{equation}
where $s_{i}$ is a continuous function  with $\|s\|_{\infty}:=\max\{\|s_{i}\|_{\infty}:~i\in \mathcal{N}\}<1$, $\|.\|_{\infty}$ is the uniform metric and $q_{i}$ is a continuous function which satisfy the following `join-up' conditions
\begin{equation}\label{1324e2}
	q_i(k_j)=p(l_{i}(k_j))-s_{i}(k_j)p(k_j)\;\;\;\;\;\text{for}\;j\in \{1,2,\dots,r\}.
\end{equation}
We consider the IFS $\mathcal{S}=\left\{(K\times  \mathbb{R},\|.\|_2),(f_{i})_{i\in \mathcal{N}}\right\}$, where
$$f_{i}(x,z)=(l_{i}(x),g_{i}(x,z))~\;\text{for}~(x,z)\in K\times  \mathbb{R}.$$
Let us suppose that  the  map $T:\mathcal{C}^*\to \mathcal{C}$ given by
$$T(f)(x)=g_{i}(l_{i}^{-1}(x),f(l_{i}^{-1}(x))~\;\;\text{for}~x\in l_i(K),i\in \mathcal{N},f\in \mathcal{C}^*$$ is well-defined, where   $\mathcal{C}=\left\{f:K\to \mathbb{R}~:~f ~\text{is continuous} \right\}$ and $\mathcal{C}^*=\\\left\{f\in  \mathcal{C}~: f|_{V_0}=p|_{V_0} \right\}$.
\begin{lemma}\label{1706r1}
	$$T(\mathcal{C}^{*})\subseteq \mathcal{C}^{**}\subseteq \mathcal{C}^{*},$$
	where  $\mathcal{C}^{**}=\left\{f\in \mathcal{C} ~: f|_{V}=p|_{V} \right\}$.
\end{lemma}
\begin{proof}
	Since 
	\begin{equation}\label{11324e1}
		T(f)(l_{i}(k_j))	=g_{i}(k_j,f(k_j))=g_{i}(k_j,p(k_j))\overset{(\ref{0702e1})\&(\ref{1324e2})}{=}p(l_{i}(k_j)),
	\end{equation} 
	for $j\in \{1,2,\dots,r\},i\in \mathcal{N}$ and $f\in \mathcal{C}^*$, we get the proof.
\end{proof}
\begin{theorem}\label{306t1}
	$T$ is a contraction map on the complete  space $(\mathcal{C}^{*},\|.\|_{\infty})$.
\end{theorem}
\begin{proof}
	For $f_1,f_2 \in \mathcal{C}^*$, we get \begin{align*}
		\|T(f_1)-T(f_2)\|_{\infty}
		&\leq \underset{i\in \mathcal{N}}{\max}~\underset{x\in K}{\max}~|g_{i}(x,f_1(x)-g_{i}(x,f_2(x)|\\
		&= \underset{i\in \mathcal{N}}{\max}~\underset{x\in K}{\max}~|s_i(x)||f_1(x)-f_2(x)|= \|s\|_{\infty}\|f_1-f_2\|_{\infty}.
	\end{align*}
\end{proof}
\begin{corollary}
	There exists a unique continuous function $f^*:K \to \mathbb{R}$ such that $f^*|_V=p|_V$ and 
	\begin{equation}\label{2101e1}
		f^*\left(l_{i}(x)\right)=s_{i}(x)f^*(x)+q_{i}(x)\;\;\; \text{for}~x\in K,i \in \mathcal{N}.
	\end{equation}
\end{corollary}
\begin{proof}
	From Theorem \ref{306t1} and the Banach contraction principle, we conclude that there exists $f^*\in \mathcal{C}^{*}$ such that $f^*=T(f^*)\overset{\text{Lemma} \ref{1706r1}}{\in} \mathcal{C}^{**}$.
\end{proof}
We call this unique map $f^*$  \textit{fractal interpolation function (FIF) on $K$}.\\
\begin{proposition}\label{167p1}
	$G_{f^*}$ is a fixed point of  the Hutchison operator $\mathcal{F}_\mathcal{S}$ of $\mathcal{S}$, where $G_{f^*}$ denotes the graph of $f^*$.
\end{proposition}
\begin{proof}
	Since $G_{f^*}\in \mathcal{H}(K\times \mathbb{R})$ and
	\begin{align}\label{47e4}
		\nonumber	G_{f^*}&=\underset{i=1}{\overset{N}{\cup}}\{\left(x, f^*(x)\right):x\in l_{i}(K)\}\overset{f^*=T(f^*)}{=}\underset{i=1}{\overset{N}{\cup}}\{(l_{i}(x), g_{i}(x,f^*(x))):x\in K\}\\
		&=\underset{i=1}{\overset{N}{\cup}}f_i(G_{f^*}),
	\end{align}
	we get the result.
\end{proof}
\begin{theorem}
	If  $\mathcal{S}$ has an attractor $A_{\mathcal{S}}$ (say), then $$A_{\mathcal{S}}=G_{f^*},$$ 
\end{theorem}
\begin{proof}
	Since $A_{\mathcal{S}}$ is the unique fixed point of the Hutchison operator of $\mathcal{S}$, from Proposition \ref{167p1}, we get the result. 
\end{proof}
\begin{lemma}
	$f_{i}:K\times [-M,M]\to K\times [-M,M],i\in \mathcal{N}$ are well-defined operators, where $M:=\frac{\underset{i\in \mathcal{N}}{\max}~\|q_{i}\|_{\infty}}{1-\|s\|_{\infty}}$.
\end{lemma}
\begin{proof}
	For $(x,z)\in K \times [-M,M]$ and $i\in \mathcal{N}$, we get
	\begin{align*}
		|g_{i}(x,z)|\leq \|s_i\|_{\infty}z+\|q_{i}\|_{\infty}\leq \|s\|_{\infty}M+\underset{i\in \mathcal{N}}{\max}~\|q_{i}\|_{\infty}= M.
	\end{align*}
\end{proof}
\begin{proposition}
	If $s_{i}$'s are  H\"{o}lder continuous, then  $\mathcal{S}$ has an attractor.
\end{proposition}
\begin{proof}
	From (\ref{2101e1}), we get
	\begin{equation}\label{0702e2}
		\|f^*\|_{\infty}\leq \|s\|_{\infty}\|f^*\|_{\infty}+\underset{i\in \mathcal{N}}{\max}~\|q_{i}\|_{\infty}~~\Rightarrow~~	\|f^*\|_{\infty}\leq M.
	\end{equation}
	From the assumption, there exist $H_{i}\in [0,\infty)$ for $i\in \mathcal{N}$ and $\eta \in (0,1]$ such that
	\begin{equation}\label{0702e11}|s_{i}(x)-s_{i}(x')|\leq H_{i}\|x-x'\|_2^{\eta}\;\;\;\text{for}\;x,x'\in K,i\in \mathcal{N}.\end{equation}
	For $\theta=\frac{1-\underset{i\in \mathcal{N}}{\max}~r_i^{\eta}}{2M \underset{i\in \mathcal{N}}{\max}~ H_{i}+1}>0,$ let us	consider the metric $d$ on $K \times \mathbb{R}$, given by
	$$d((x,z),(x',z'))=\|x-x'\|_2^{\eta}+\theta|(z-f^*(x))-(z'-f^*(x'))|,$$
	for $(x,z),(x',z')\in K\times \mathbb{R}$, which is equivalent to the Euclidean metric.\\
	For $i \in \mathcal{N}$ and $(x,z),(x',z')\in K\times [-M,M]$, we have
	\begin{align*}
		&d(f_{i}(x,z),f_i(x',z'))\\
		&=\|l_i(x)-l_{i}(x')\|_2^{\eta}+\theta|s_{i}(x)z+q_{i}(x)-f^*(l_{i}(x)))-(s_{i}(x')z'+q_{i}(x')-f^*(l_{i}(x')))|\\
		&\overset{(\ref{2101e1})}{=} r_i^{\eta}\|x-x'\|_2^{\eta}+\theta|s_{i}(x)(z-f^*(x))-s_{i}(x')(z'-f^*(x')))|\\
		&\leq r_i^{\eta}\|x-x'\|_2^{\eta}+\theta\|s_{i}\|_{\infty}|(z-f^*(x))-(z'-f^*(x'))|\\
		&~~~+\theta|(z'-f^*(x'))||s_{i}(x)-s_{i}(x')|\\
		& \overset{(\ref{0702e2})\&(\ref{0702e11})}{\leq} r_i^{\eta}\|x-x'\|_2^{\eta}+\theta\|s\|_{\infty}|(z-f^*(x))-(z'-f^*(x'))|+\theta2M H_{i}\|x-x'\|_2^{\eta}\\
		&\leq c_id((x,z),(x',z')),
	\end{align*}
	where  $c_i=\max \left\{r_i^{\eta}+\theta2M H_{i},\|s\|_{\infty}\right\}<1$.\\
	Thus,  $f_{i}$'s are contractions on $(K\times [-M,M],d)$.\\
	From Remark \ref{1thm1}, we get the result.
\end{proof}

\begin{case}\label{1324ca1}
	If $K$ is a p.c.f. self-similar set and $V_0$ is the boundary of $K$,  then  by Remark \ref{131re1} (ii), we get 
	\begin{equation}\label{0407e1}
		l_{i}(K)\cap l_{i'}(K)=l_{i}(V_0)\cap l_{i'}(V_0)~~\text{for}~~i,i'\in \mathcal{N}~\text{with}~i\neq i'.
	\end{equation}
	From equation  (\ref{11324e1}), we get
	\begin{equation}\label{0407e2}
		T(f)|_{l_i(V_0)}=p|_{l_i(V_0)}~~\text{for}~f\in \mathcal{C}^*,i\in \mathcal{N}.
	\end{equation}
	Let $f\in \mathcal{C}^*$ and  $x\in l_{i}(K)\cap l_{i'}(K)$ for some $i,i'\in \mathcal{N}$ with $i\neq i'$.\\
	From equation (\ref{0407e1}) and (\ref{0407e2}), we get	
	$$T(f)(x)=p(x)$$
	by viewing $x$ as an entity belonging to $l_{i}(K)$ and  $l_{i'}(K)$. \\
	Therefore $T(f)$ is a continuous function. Thus, $T$ is well defined.
\end{case}
\begin{remark}If $K$ is a p.c.f. self-similar set, then for any given data on it's boundary $V_0$,  there exists a unique harmonic function on $K$ such that it satisfies the given data (see \cite{KI}). This gives the guarantee for the existence of $q_i$'s as in (\ref{1324e2}).\end{remark}
\begin{remark}\label{13324r1}
	If $K$ is a p.c.f. self-similar set and $V_0$  is  its boundary with respect to the IFS  $\{(\mathbb{R}^m,\|.\|_2),(l_i)_{i=1}^N\}$, then  for any $n\in \mathbb{N}$, $K$ $(V_0)$  is again a p.c.f. self-similar set (boundary of $K$)  with respect to the IFS
	$\{(\mathbb{R}^2,\|.\|_2),\{l_{\omega}\}_{\omega \in \mathcal{N}^n}\}$.  So, for any $n\in \mathbb{N},$ we can get a FIF $f^*$ for the given data on $(V=)V_n:=\underset{\omega\in \mathcal{N}^n}{{\cup}}l_{\omega}(V_0)$. Note that $V_n\to K$ as $n\to \infty$. 	
\end{remark}
\begin{remark}
	Sierpi\'nski Gasket (SG),  Sierpi\'nski sickle, Koch curve, Hata's tree-like set are some of the  examples of p.c.f. self-similar sets.\\
	The IFS of SG is $\{(\mathbb{R}^2,\|.\|_2),\{l_i\}_{i=1}^3\}$, where $l_i(x)=\frac{1}{2}(x+k_i)$ for $i\in \{1,2,3\}$ and $\{k_1,k_2,k_3\}$ are the vertices of an equilateral triangle. In this case,  the boundary of $K$ is $V_0=\{k_1,k_2,k_3\}$.
\end{remark}
\begin{case}\label{1324ca2}
	If $$V=\left\{(x_{1i_1},x_{2i_2},\dots,x_{mi_m})\in \mathbb{R}^{m}:i_u\in \{0,1,\dots,n_u\},u\in \{1,2,\dots,m\}\right\}$$ with $x_{u0}<x_{u1}<\dots <x_{un_u}$ for $u\in \{1,2,\dots,m\}$, then we define $l_i:\mathbb{R}^{m}\to \mathbb{R}^{m}$ to be
	$$l_{i}(x)=\left(l_{1i_1}(x_1),l_{2i_2}(x_2),\dots,l_{mi_m}(x_m)\right),$$
	for $x=(x_1,x_2,\dots,x_m)\in \mathbb{R}^{m}$ and $i=(i_1,i_2,\dots,i_m)\in \mathcal{N},$ where 
	\begin{itemize}
		\item[$-$] $\mathcal{N}=\{(i_1,i_2,\dots i_m):i_u\in \{1,2,\dots,n_u\},u\in \{1,2,\dots,m\}\}$,
		\item[$-$] $l_{ui_u}(t)=\left(\frac{x_{u(i_u-\epsilon_{ui_u})}-x_{u(i_u-1+\epsilon_{ui_u})}}{x_{un_u}-x_{u0}}\right)t+\left(\frac{x_{u(i_u-1+\epsilon_{ui_u})}x_{un_u}-x_{u(i_u-\epsilon_{ui_u})}x_{u0}}{x_{un_u}-x_{u0}}\right)$ for $t\in \mathbb{R}, u\in \{1,2,\dots,m\}$ and
		\item[$-$] $\epsilon_u=(\epsilon_{u1},\epsilon_{u2},\dots,\epsilon_{un_u})\in \{0,1\}^{n_u}$  for $u\in \{1,2,\dots,m\}$ (which is called signature).
	\end{itemize}   
	We defined $l_{ui_u}$ for $u\in \{1,2,\dots,m\}$ such that 
	$$l_{ui_u}([x_{u0},x_{un_u}])=[x_{u(i_u-1)},x_{ui_u}]$$
	and $$l_{ui_u}(x_{u0})=x_{u(i_u-1+\epsilon_{ui_u})}~~\text{and}~~~l_{ui_u}(x_{un_u})=x_{u(i_u-\epsilon_{ui_u})}.$$
	From our construction, we get 
	\begin{itemize}
		\item[$-$] $V_0=\left\{(x_{1i_1},x_{2i_2},\dots,x_{mi_m})\in \mathbb{R}^{m}:i_u\in \{0,n_u\},u\in \{1,2,\dots,m\}\right\}$ and
		\item[$-$] $K=[x_{10},x_{1n_1}]\times [x_{20},x_{2n_2}]\times \dots \times [x_{m0},x_{mn_m}]$.
	\end{itemize}	  
	\begin{itemize}
		\item[(i)] If $m=1$, then $K$ is a p.c.f. self-similar set and $V_0$ is its boundary. So, $T$ is well-defined. \\
		In this case, we obtain the fractal interpolation function of the zipper on an interval (see \cite{cvvt}).\\
		If we choose $\varepsilon_u=0$ for	$u\in \{1,2,\dots,m\}$ and $s_i$'s are constant functions, we get the  standard fractal interpolation function on an interval (see \cite{b1}).\\
		\item[(ii)]	If $m> 1$, then 	$T$ is well-defined provided that
		$$\varepsilon_u=(0,1,0,1,\dots)~\text{or}~\varepsilon_u= (1,0,1,0,\dots)~\text{for}~	u\in \{1,2,\dots,m\}$$ 
		and 
		$$	F_{i}(x,z)=F_{(	i_1\dots i_{j-1},i_{j}+1,i_{j+1}\dots i_m)}(x,z),$$
		for $i=(i_1,i_2\dots i_m)\in \mathcal{N},j\in \{1,2\dots m\}$ with $i_j\in \{1,2\dots n_j-1\}$,\\ $x=(x_1\dots x_{j-1},x_j^*,x_{j+1}\dots x_m)\in K$ with $x_j^* = l_{ji_j}^{-1}(x_{ji_j})=l_{j(i_j+1)}^{-1}(x_{ji_j})$ and $z\in \mathbb{R}$. \\
		For more details ref. \cite{mass,fis-13,KUM}.		
		In this case, we call $f^*$  \textit{multivariate FIF}. \\
		In particular, if $\varepsilon_u=(0,1,0,1,\dots)$
		for  $u\in  \{1,2,\dots,m\}$, $s_i$'s are equal constants i.e., there exists unique $s\in(-1,1)$ such that $s_i(x)=s$ for $x\in K$ and $i\in \mathcal{N}$, and	$$q_i(x)=\underset{J\subseteq \{1,2,\dots,m\}}{\sum}e_{i,J}\underset{j\in J}{\prod}x_j~~~\text{for}~x=(x_1,x_2,\dots,x_m) \in K,i\in \mathcal{N},$$
		where $e_{i,J}$'s are constants, then $T$ is well-defined (see \cite{MAL}).

	\end{itemize}
\end{case}

\section{Box dimension of  fractal  functions}\label{2401s4}
For a continuous function $f:K \to \mathbb{R}$ and $\omega\in \mathcal{N}^k$, we define the oscillation of $f$ over $l_{\omega}(K)$ by
\begin{align*}
	\text{Osc}_{\omega}(f)=\underset{x\in l_{\omega}(K)}{\sup}f(x)-\underset{x\in l_{\omega}(K)}{\inf}f(x)=\underset{x,x' \in l_{\omega}(K)}{\sup}|f(x)-f(x')|
\end{align*}
and total oscillation of
order $k$ by
$$\text{Osc}(k,f) =\underset{\omega \in \mathcal{N}^k}{\sum} \text{Osc}_{\omega}(f).$$
We  define oscillation space on $K$, for $\eta\in \left[0,\log_{\Lambda} N\right]$,  as follows (Ref. \cite{DEL})
$$\mathcal{C}^{\eta}(K)= \{f : K \to \mathbb{R}: f ~\text{is continuous and}~[f]_{{\eta}} < \infty\},$$
where $[f]_{{\eta}}=\underset{k\in \mathbb{N}}{\sup} \frac{\text{Osc}(k,f)}{\Lambda^{k\left(\log_{\Lambda} N-\eta\right)}}$ and $\Lambda^{-1}=\underset{i\in \mathcal{N}}{\max}~r_i$.
\begin{remark}
	If $m=1$ and $\eta=\log_{\Lambda} N,$ then $\mathcal{C}^{\eta}(K)$ is the collection of continuous bounded variation functions on $K$.
\end{remark}
In the following proposition, we discuss the relation between H\"{o}lder continuous functions and oscillation spaces, draws inspiration from \cite{car}.
\begin{proposition}\label{2612prop1}
	If $f:K\to \mathbb{R}$ is a  H\"{o}lder continuous function with exponent $\eta \in (0,1],$ then $f\in \mathcal{C}^{\eta}(K).$
\end{proposition}
\begin{proof}
	By assumption, there exists a constant $H$ such that
	$$|f(x)-f(x')|\leq H \|x-x'\|_2^{\eta}\;\;\;\text{for}~~x,x'\in K.$$
	Thus, we have
	\begin{align*}
		\text{Osc}(k,f)&=\underset{\omega \in \mathcal{N}^k}{\sum}\underset{x,x' \in l_{\omega}(K)}{\sup}|f(x)-f(x')|\leq H \underset{\omega \in \mathcal{N}^k}{\sum}\underset{x,x' \in K}{\sup} \|l_{\omega}(x)-l_{\omega}(x')\|_2^{\eta}\\
		&\leq H \underset{\omega \in \mathcal{N}^k}{\sum}\left(\Lambda^{-k}\right)^{\eta}\underset{x,x' \in K}{\sup} \|x-x'\|_2^{\eta}=H  N^k\Lambda^{-k\eta}|K|^{\eta}\\
		&=H |K|^{\eta} \Lambda^{k\left(\log_{\Lambda} N-\eta\right)},
	\end{align*}
	for all $k\in \mathbb{N}$, where $|K|$ denotes the diameter of $K$.
\end{proof}
\begin{theorem}	\label{912thm1}
	Let $s_{i},q_{i}\in \mathcal{C}^{\eta}(K)$ for $i \in \mathcal{N}$. If:
	\begin{itemize}
		\item[(i)]  $\gamma\leq \frac{N}{\Lambda^{\eta'}},$ then $$ \dim_H(K)\leq \dim_H \left(G_{f^*}\right)\leq \underline{\dim}_B\left(G_{f^*}\right)\leq\overline{\dim}_B\left(G_{f^*}\right)\leq 1-\eta'+\log_{\Lambda}N;$$
		\item[(ii)]  $\gamma> \frac{N}{\Lambda^{\eta'}},$ then $$ \dim_H(K)\leq \dim_H \left(G_{f^*}\right)\leq \underline{\dim}_B\left(G_{f^*}\right)\leq\overline{\dim}_B\left(G_{f^*}\right)\leq 1+\log_{\Lambda} \gamma,$$
	\end{itemize}
	where $\gamma:=\underset{i\in \mathcal{N}}{\sum}\|s_{i}\|_{\infty}$ and $\eta':=\min \{1,\eta\}$.		
\end{theorem}
\begin{proof}
	Let ${N}(k)$ and ${N}(k,\omega)$ denote the minimum number of cubes of size $\frac{|K|}{\Lambda^{k}}\times \frac{|K|}{\Lambda^{k}} \times\dots  \frac{|K|}{\Lambda^{k}}$ which covers $G_{f^*}$  and   $G_{f^*,\omega}$ respectively, where $G_{f^*,\omega}=\{(x,f^*(x)):x\in l_{\omega}(K)\}$ for $\omega \in \mathcal{N}^k$. \\
	For $i\in \mathcal{N},\omega\in \mathcal{N}^k,k\in \mathbb{N}$ and $x,x'\in l_{\omega}(K)$, we have
	\begin{align*}
		&|f^*(l_{i}(x))-f^*(l_{i}(x'))|\\
		&\overset{(\ref{2101e1})}{\leq}\|s_{i}\|_{\infty}|f^*(x)-f^*(x')|+\|f^*\|_{\infty}|s_{i}(x)-s_{i}(x')|+|q_{i}(x)-q_{i}(x')|\\
		&\leq \|s_{i}\|_{\infty} \frac{{N}(k,\omega)|K|}{\Lambda^k}+ \|f^*\|_{\infty}	\text{Osc}_{\omega}(s_{i})+	\text{Osc}_{\omega}(q_{i}).
	\end{align*}	
	For $i\in \mathcal{N},\omega\in \mathcal{N}^k$ and $k\in \mathbb{N}$, we get	$G_{f^*,(i,\omega)}$ is contained in a cuboid of size $\frac{|K|}{\Lambda^{k+1}}\times \frac{|K|}{\Lambda^{k+1}}\times \dots  \frac{|K|}{\Lambda^{k+1}}\times  \frac{\|s_{i}\|_{\infty}{N}(k,\omega)|K|}{\Lambda^k}+ \|f^*\|_{\infty}	\text{Osc}_{\omega}(s_{i})+	\text{Osc}_{\omega}(q_{i})$.\\
	Since $\Lambda>1$ and
	\begin{equation}\label{2424eq1}
		G_{f^*}\overset{(\ref{47e4})}{=}\underset{\omega \in \mathcal{N}^k}{\cup}G_{f^*,\omega}~\;\;\text{for}~k\in \mathbb{N},
	\end{equation}
	we get
	\begin{align*}
		&{N}(k+1)\leq \underset{\omega\in \mathcal{N}^{k+1}}{\sum}{N}(k+1,\omega)=\underset{i\in \mathcal{N}}{\sum}\underset{\omega\in \mathcal{N}^k}{\sum}{N}(k+1,(i,\omega))\\
		& ~~ \leq \underset{i\in \mathcal{N}}{\sum}\underset{\omega\in \mathcal{N}^k}{\sum}\left(\frac{\Lambda^{k+1}}{|K|}\left( \frac{\|s_{i}\|_{\infty}{N}(k,\omega)|K|}{\Lambda^k}+ \|f^*\|_{\infty}	\text{Osc}_{\omega}(s_{i})+	\text{Osc}_{\omega}(q_{i})\right)+1\right)\\
		& ~~ \leq \Lambda\gamma {N}(k)+ \frac{\Lambda^{k+1}}{|K|}\underset{i\in \mathcal{N}}{\sum}\left(\|f^*\|_{\infty}	[s_{i}]_{\eta}+	[q_{i}]_{\eta}\right)\Lambda^{k\left(\log_{\Lambda} N-\eta\right)}+N^{k+1}\\
		& ~~ \leq \Lambda\gamma{N}(k)+ N^k\Lambda^{k(1-\eta')}C,
	\end{align*}
	for  $k\in \mathbb{N}$,	where $C=\frac{\Lambda}{|K|}\underset{i\in \mathcal{N}}{\sum}\left(\|f^*\|_{\infty}	[s_{i}]_{\eta}+	[q_{i}]_{\eta}\right)+N$.\\
	Via the mathematical induction method,  for $k\in \mathbb{N},$ we  get
	\begin{align}\label{2201e5}
		\nonumber	&{N}(k+1)\leq(\Lambda\gamma)^2{N}(k-1)+\Lambda\gamma\left(N\Lambda^{1-\eta'}\right)^{k-1}C+\left(N\Lambda^{1-\eta'}\right)^kC\\
		\nonumber	&~\leq (\Lambda\gamma)^3{N}(k-2)+C\left((\Lambda\gamma)^2\left(N\Lambda^{1-\eta'}\right)^{k-2}+\Lambda\gamma\left(N\Lambda^{1-\eta'}\right)^{k-1}+\left(N\Lambda^{1-\eta'}\right)^k\right)\\
		\nonumber	&\leq(\Lambda\gamma)^k{N}(1)+C\left((\Lambda\gamma)^{k-1}N\Lambda^{1-\eta'}+(\Lambda\gamma)^{k-2}\left(N\Lambda^{1-\eta'}\right)^2+\dots +\left(N\Lambda^{1-\eta'}\right)^k\right).
	\end{align}
	$\mathit{Case}~(i).$ Consider  $ \gamma\leq \frac{N}{\Lambda^{\eta'}}$, then for $k\in \mathbb{N}$, we have
	\begin{align*}
		&{N}(k+1)\leq (\Lambda\gamma)^k{N}(1)+C\left(N\Lambda^{1-\eta'}\right)^k\left(1+ \frac{\Lambda^{\eta'}\gamma}{N}+\dots+\left(\frac{\Lambda^{\eta'}\gamma}{N}\right)^{k-1}\right)\\
		& ~~\leq \Lambda^k\left(\frac{N}{\Lambda^{\eta'}}\right)^k{N}(1)+C\left(N\Lambda^{1-\eta'}\right)^kk\\
		& ~~\leq \left(N\Lambda^{1-\eta'}\right)^{k+1} (k+1)(N(1)+C).
	\end{align*} 
	Therefore, we get
	\begin{align*}
		&	\overline{\dim}_B(G_{f^*})= \limsup_{k \to \infty}\frac{\log {N}(k+1) }{\log \left(\frac{\Lambda^{k+1}}{|K|}\right)}\leq 1-\eta'+\frac{\log N}{\log \Lambda}.
	\end{align*}
	$\mathit{Case}~(ii).$ Consider  $\gamma> \frac{N}{\Lambda^{\eta'}},$ then for  $k\in \mathbb{N}$, we have
	\begin{align*}
		&{N}(k+1) \leq(\Lambda\gamma)^k{N}(1)+C(\Lambda\gamma)^{k-1}\Lambda\left(\frac{N}{\Lambda^{\eta'}}\right)\left(1+\frac{N}{\Lambda^{\eta'}\gamma}+\dots+\left(\frac{N}{\Lambda^{\eta'}\gamma}\right)^{k-1}\right)\\
		&\leq(\Lambda\gamma)^k{N}(1)+C(\Lambda\gamma)^{k-1}\Lambda\gamma \frac{1}{1-\frac{N}{\Lambda^{\eta'}\gamma}}\\	&\leq(\Lambda\gamma)^{k+1}\left({N}(1)+\frac{C}{1-\frac{N}{\Lambda^{\eta'}\gamma}}\right).
	\end{align*} 
	Hence
	\begin{align*}
		\overline{\dim}_B(G_{f^*})\leq 1+\log_{\Lambda} \gamma.
	\end{align*}
\end{proof}
Let us denote $\Lambda_0^{-1}=\underset{i\in \mathcal{N}}{\min}~r_i$, and for $r\in \{0,1,2,\dots,m\}$, we denote:
\begin{itemize}	
	\item[$-$]  $\gamma_{1,r}=\underset{i\in \mathcal{N}}{\sum}s_{i,1,r}$, where
	$$s_{i,1,r}=\begin{cases}
		|s_{i}|, &\quad 
		\text{if~} r=0, s_{i} \text{~is a constant and~} q_{i} \text{~is affine} \\
		|s_{i}|, &\quad 
		\text{if~} r\neq 0, s_{i} \text{~is a constant and~}\\		&\quad q_{i} \text{~is affine with respect to the $r^{\text{th}}$ co-ordinate ~} \\
		0, &\quad \text{~~otherwise;}
	\end{cases}$$
	\item[$-$]  $\gamma_{2,r}=\underset{i\in \mathcal{N}}{\sum}s_{i,2,r}$, where
	$$s_{i,2,r}=\begin{cases}
		s_{i}, &\quad 
		\text{if~} r=0,	s_{i}  \text{~is a non-negative constant and~} q_{i} \text{~is concave} \\
		s_{i}, &\quad 
		\text{if~}r\neq 0,	s_{i}  \text{~is a non-negative constant and~}\\
		&\quad q_{i} \text{~is concave with respect to the $r^{\text{th}}$ co-ordinate~} \\
		0, &\quad \text{~~otherwise;}
	\end{cases}$$
	\item[$-$]   $\gamma_{3,r}=\underset{i\in \mathcal{N}}{\sum}s_{i,3,r}$, where
	$$s_{i,3,r}=\begin{cases}
		s_{i}, &\quad 
		\text{if~}r=0,	s_{i}  \text{~is a non-negative constant and~} q_{i} \text{~is convex} \\
		s_{i}, &\quad 
		\text{if~}r\neq 0,	s_{i}  \text{~is a non-negative constant and~}\\
		&\quad q_{i} \text{~is convex with respect to the $r^{\text{th}}$ co-ordinate~} \\
		0, &\quad \text{~~otherwise.}
	\end{cases}$$
\end{itemize}
\begin{lemma}\label{1701l1}
	Assume that there exist $y_1,y_2,y_3\in V$ and $\lambda\in (0,1)$ such that  $y_3=(1-\lambda)y_1+\lambda y_2$ and the set $\{(y_1,p(y_1)),(y_2,p(y_2)),(y_3,p(y_3))\}$ is not collinear i.e., $L:=p(y_3)-((1-\lambda)p(y_1)+\lambda p(y_2))\neq 0$. 	Then for $\omega\in \mathcal{N}^k,k\in \mathbb{N}$, we get
	\begin{itemize}
		\item[(i)] $G_{f^*,\omega}$  must cover a  set of height  $s_{\omega_1,1,0}s_{\omega_2,1,0}\dots s_{\omega_k,1,0}|L|$;
		\item[(ii)] $G_{f^*,\omega}$  must cover a  set of height  $s_{\omega_1,2,0}s_{\omega_2,2,0}\dots s_{\omega_k,2,0}L$ if  $L>0$;  
		\item[(iii)] $G_{f^*,\omega}$  must cover a  set of height  $s_{\omega_1,3,0}s_{\omega_2,3,0}\dots s_{\omega_k,3,0}|L|$ if $L<0$.
	\end{itemize}	
\end{lemma}
\begin{proof}
	Let us suppose that  $L>0$.\\	   	
	Since $l_i$'s are similarities on $\mathbb{R}^m$, they are affine maps.\\
	Thus, for $\omega \in \mathcal{N}^k,k\in \mathbb{N}$, we get
	\begin{equation}\label{1701e2}
		l_{\omega}(y_3)=(1-\lambda)l_{\omega}(y_1)+\lambda l_{\omega}(y_2).
	\end{equation}
	For $i \in \mathcal{N}$ such that $s_{i}$  is a non-negative constant and $q_{i}$ is concave, we get
	\begin{align}\label{81eq1}
		\nonumber& g_{i}(y_3,p(y_3))- ((1-\lambda)g_{i}(y_1,p(y_1))+\lambda g_{i}(y_2,p(y_2)))\\
		&~~~= s_i L+q_{i}((1-\lambda)y_1+\lambda y_2)-((1-\lambda)q_{i}(y_1)+\lambda q_{i}(y_2))\geq s_i L.
	\end{align}
	For $i,j\in \mathcal{N}$ such that $s_{i},s_j$ are non-negative constants and  $q_{i},q_{j}$ are concave, we get
	\begin{align*}
		& g_{i}(f_{j}(y_3,p(y_3)))- ((1-\lambda) g_{i}(f_{j}(y_1,p(y_1)))+\lambda  g_{i}(f_{j}(y_2,p(y_2))))\\ 
		&~\overset{(\ref{1701e2})}{=}s_{i}(g_{j}(y_3,p(y_3))- ((1-\lambda)g_{j}(y_1,p(y_1))+\lambda g_{j}(y_2,p(y_2))))\\
		&~\;\;+q_{i}((1-\lambda)l_{j}(y_1)+\lambda l_j(y_2))-((1-\lambda)q_{i}(l_{j}(y_1))+\lambda q_{i}(l_{j}(y_2)))\\
		&~\;\overset{(\ref{81eq1})}{\geq} s_is_jL.
	\end{align*}
	Via the mathematical induction method, for $i\in \mathcal{N}$ and $\omega \in \mathcal{N}^{k-1},k-1\in \mathbb{N}$, we get
	\begin{align}\label{1701e3}
		\nonumber	&|g_{i}(f_{\omega}(y_3,p(y_3)))- ((1-\lambda) g_{i}(f_{\omega}(y_1,p(y_1)))+\lambda  g_{i}(f_{\omega}(y_2,p(y_2))))|\\
		&	\;\geq s_{i,2,0}s_{\omega_1,2,0}s_{\omega_2,2,0}\dots s_{\omega_k,2,0}L.
	\end{align}
	Since $f^*$ is a continuous function passing through $f_{\omega}(y_1,p(y_1)),f_{\omega}(y_2,p(y_2))$ and $f_{\omega}(y_2,p(y_2))$, and by using (\ref{1701e2}) and (\ref{1701e3}), we get  $G_{f^*,\omega}$  must cover a set of height  $s_{\omega_1,2,0}s_{\omega_2,2,0}\dots s_{\omega_k,2,0}L$, for  $\omega\in \mathcal{N}^k, k\in \mathbb{N}$.\\
	In a similar way, we can prove the other cases also. 
\end{proof}
\begin{lemma}\label{1701l12}
	Let us consider $l_i(x)=(l_{1i_1}(x_1),l_{2i_2}(x_2),\dots,l_{mi_m}(x_m))$ for $x\in \mathbb{R}^m$ and $i\in \mathcal{N}$, where $l_{ui_u}:\mathbb{R}\to \mathbb{R}$ are affine maps for $u\in \{1,2,\dots,m\}$. 	Assume that there exist $y_1=(t_{1},t_2,\dots,t_{r-1},t_{r_1},t_{r+1},\dots,t_{m}),  y_2=(t_{1},t_2,\dots,t_{r-1},t_{r_2},t_{r+1},\dots,t_{m}),$ $y_3=(t_{1},t_2,\dots,t_{r-1},t_{r_3},t_{r+1},\dots,t_{m})\in V$ for some $r\in \{1,2,\dots,m\}$, and $\lambda\in (0,1)$ such that  $y_3=(1-\lambda)y_1+\lambda y_2$ and  $L=p(y_3)-((1-\lambda)p(y_1)+\lambda p(y_2))\neq 0$. 	Then for $\omega\in \mathcal{N}^k,k\in \mathbb{N}$, we get
	\begin{itemize}
		\item[(i)] $G_{f^*,\omega}$  must cover  a  set of height  $s_{\omega_1,1,r}s_{\omega_2,1,r}\dots s_{\omega_k,1,r}|L|$;
		\item[(ii)] $G_{f^*,\omega}$  must cover  a set of height  $s_{\omega_1,2,r}s_{\omega_2,2,r}\dots s_{\omega_k,2,r}L$ if  $L>0$;  
		\item[(iii)] $G_{f^*,\omega}$  must cover  a  set of height  $s_{\omega_1,3,r}s_{\omega_2,3,r}\dots s_{\omega_k,3,r}|L|$ if $L<0$.
	\end{itemize}
\end{lemma}
\begin{proof}
	Using similar arguments of Lemma \ref{1701l1}, we can prove this result. 
\end{proof}
\begin{theorem}\label{2301t1}
	Let $f^*$ be a multivariate FIF derived from Case \ref{1324ca2}.   Assume that the interpolation points are not collinear, i.e., 
	there exist $r \in \{1,2,\dots,m\},y_1,y_2,y_3\in V$ and  $L\neq 0$ as in the framework of Lemma \ref{1701l12}. If:
	\begin{itemize}
		\item[(i)]  $\gamma_{1,r}\neq 0$, then
		$$1+\log_{\Lambda_0} (\gamma_{1,r}) \leq \underline{\dim}_B\left(G_{f^*}\right);$$
		\item[(ii)] $L>0$ and $\gamma_{2,r}\neq0$, then 
		$$1+\log_{\Lambda_0} (\gamma_{2,r}) \leq \underline{\dim}_B\left(G_{f^*}\right);$$
		\item[(iii)]   $L<0$ and $\gamma_{3,r}\neq0$, then 
		$$1+\log_{\Lambda_0} (\gamma_{3,r}) \leq \underline{\dim}_B\left(G_{f^*}\right).$$
	\end{itemize}
\end{theorem}
\begin{proof}
	Let ${N}_0(k)$ denote the minimum number of cubes of side length $\frac{|K|_0}{\Lambda_0^{k}}$ that covers $G_{f^*}$, where $|K|_0$ is the minimum of the side lengths of $K$. \\
	Let us suppose that $\gamma_{1,r}\neq 0$.\\
	From Lemma \ref{1701l12}, we get $G_{f^*,\omega}$ must cover  a cuboid of size $\frac{|K|_0}{\Lambda_0^k}\times \dots\frac{|K|_0}{\Lambda_0^k}\times s_{\omega_1,1,r}s_{\omega_2,1,r}\dots s_{\omega_k,1,r}|L|$, for $\omega \in \mathcal{N}^k,k\in \mathbb{N}$.\\
	Thus
	\begin{equation}\label{2301e2}
		{N}_0(k)\geq \underset{\omega \in\mathcal{N}^k}{\sum} \frac{\Lambda_0^k}{|K|_0}s_{\omega_1,1,r}s_{\omega_2,1,r}\dots s_{\omega_k,1,r}|L|=\frac{\Lambda_0^k}{|K|_0}\gamma_{1,r}^k|L|~~~\text{for}~  k\in \mathbb{N}.
	\end{equation}
	Therefore
	\begin{align*}
		\underline{\dim}_B(G_{f^*})= \liminf_{k \to \infty}\frac{\log {N}_0(k) }{\log \left(\frac{\Lambda_0^{k}}{|K|_0}\right)}\geq1+\frac{\log \gamma_{1,r}}{\log \Lambda_0}.
	\end{align*}
	A similar argument works for other cases as well.
\end{proof}
\begin{corollary}\label{3001cor1}
	Let  $f^*$ be a multivariate FIF derived from Case \ref{1324ca2}, $q_{i}\in \mathcal{C}^{\eta}(K)$ and $s_{i}$ be a constant for $i \in \mathcal{N}$. Assume that the interpolation points are not collinear (i.e., there exists $r\in \{1,2,\dots,m\}$ as in Lemma \ref{1701l12})  
	\begin{itemize}
		\item[$-$] and either  $q_{i}$ is affine with respect to the $r^{\text{th}}$ co-ordinate for $i \in \mathcal{N}$\\
		or
		\item[$-$] $L>0$, $q_{i}$ is concave with respect to the $r^{\text{th}}$ co-ordinate  and $s_{i}\geq 0$ for $i \in \mathcal{N}$\\
		or
		\item[$-$]  $L<0$,  $q_{i}$ is convex  with respect to the $r^{\text{th}}$ co-ordinate  and $s_{i}\geq 0$ for $i \in \mathcal{N}$.
	\end{itemize}
	If:
	\begin{itemize}
		\item [(i)]   $0<\gamma\leq \frac{N}{\Lambda^{\eta'}},$ then $$1+\log_{\Lambda_0}\gamma\leq  \underline{\dim}_B\left(G_{f^*}\right)\leq \overline{\dim}_B\left(G_{f^*}\right)\leq 1-\eta'+\log_{\Lambda}N;$$
		\item[(ii)]  $\gamma> \frac{N}{\Lambda^{\eta'}},$ then $$1+\log_{\Lambda_0}\gamma\leq  \underline{\dim}_B\left(G_{f^*}\right)\leq \overline{\dim}_B\left(G_{f^*}\right)\leq 1+\log_{\Lambda} \gamma;$$
		\item[(iii)] $n_u=n$ and $\{x_{ui_u}\}_{i_u=0}^{n_u}$ are equally spaced points for $u\in \{1,2,\dots,m\}$, then
		$$\dim_B\left(G_{f^*}\right)=\begin{cases}
			1+\frac{\log \gamma}{\log n},  	&\quad\text{if~} 	\gamma> n^{m-\eta'},\\
			m, &\quad\text{if}~ \gamma\leq n^{m-1} ~\text{and}~ \eta'=1.
		\end{cases}$$
	\end{itemize}
\end{corollary}
\begin{theorem}\label{2101t1}
	Let $n\in \mathbb{N}$ and  $f^*$ be a FIF on SG with respect to the data $p$ on $V=V_n$ as in Remark \ref{13324r1}. Assume that there exist $y_1,y_2,y_3\in V$ and  $L\neq 0$ as in the framework of Lemma \ref{1701l1}. If:
	\begin{itemize}
		\item[(i)] $\gamma_{1,0}\neq0$, then
		$$1+\log_{2^n} \gamma_{1,0} \leq \underline{\dim}_B\left(G_{f^*}\right);$$
		\item[(ii)]  $L>0$  and $\gamma_{2,0}\neq0$, then 
		$$1+\log_{2^n} \gamma_{2,0} \leq \underline{\dim}_B\left(G_{f^*}\right);$$
		\item[(iii)]  $L<0$  and $\gamma_{3,0}\neq0$, then
		$$1+\log_{2^n} \gamma_{3,0} \leq \underline{\dim}_B\left(G_{f^*}\right);$$
		\item[(iv)] $s_{i},q_{i}\in \mathcal{C}^{\eta}(\text{SG})$  for $i \in \mathcal{N}$ and either $\gamma_{1,0}=\gamma$ (or) $L>0$  and $\gamma_{2,0}=\gamma$ (or) $L<0$  and $\gamma_{3,0}=\gamma$, then		
		$$\dim_B\left(G_{f^*}\right)=\begin{cases}
			1+\frac{\log \gamma}{\log 2^n}, &\quad 	\text{if~} \gamma>\left(\frac{3}{2^{\eta'}}\right)^n,\\
			\frac{\log 3}{\log 2}, &\quad \text{if~} 	\gamma\leq \left(\frac{3}{2}\right)^n \text{~and~} \eta'=1.
		\end{cases}$$
	\end{itemize}
\end{theorem}
\begin{proof}
	From the assumption,	 we get $N=3^n,r_{\nu}=\frac{1}{2^n}$ for  $\nu\in \mathcal{N}=\{1,2,3\}^n$ and $\Lambda_0=\Lambda=2^n$.\\
	Let ${N}_{\text{S}}(k)$ and ${N}_{\text{S}}(k,\omega)$ denote the minimum number  of sets belonging to the family
	$$\left\{T\times \left[z,z+\frac{|K|}{2^{nk}}\right]: T ~\text{is an equilateral triangle of side length~} \frac{|K|}{2^{nk}} \right\},$$
	which  covers $G_{f^*}$ and $G_{f^*,\omega}$ respectively. \vspace{1.5mm}\\
	Let us suppose that $\gamma_{1,0}\neq0$.\\	
	From Lemma  \ref{1701l1}, for $k\in \mathbb{N},$ we get   	
	\begin{equation}\label{7324e1}
		{N}_{\text{S}}(k)=\underset{\omega\in \mathcal{N}^k}{\sum}{N}_{\text{S}}(k,\omega)\geq \underset{\omega\in \mathcal{N}^k}{\sum} \frac{2^{nk}}{|K|}s_{\omega_1,1,0}s_{\omega_2,1,0}\dots s_{\omega_k,1,0}|L|=\frac{2^{nk}\gamma_{1,0}^k |L|}{|K|}.
	\end{equation}
	Since 	$${N}_{ \frac{|K|}{2^{nk}}}(G_{f^*})\leq {N}_{\text{S}}(k) \leq 3{N}_{ \frac{|K|}{2^{nk}}}(G_{f^*})~~\text{for}~k\in \mathbb{N},$$
	we have 
	\begin{align*}
		\underline{\dim}_B(G_{f^*})=\liminf_{k \to \infty}\frac{\log {N}_{\text{S}}(k) }{\log \left(\frac{2^{nk}}{|K|}\right)}\overset{(\ref{7324e1})}{\geq} 1+\log_{2^n} \gamma_{1,0}.
	\end{align*}		
	In a similar way, we can  prove the  other cases.
\end{proof}	

\begin{remark}By using Lemma \ref{1701l1} and \ref{1701l12}, in a similar way,  we can estimate non-trivial lower bounds of the box dimension of FIFs  on many different domains such as triangle,  Sierpi\'nski sickle etc.\end{remark}

\begin{theorem}	\label{912thm111}
	Let  $f^*$ be a  FIF on an interval derived from Case \ref{1324ca2},  $s_{i},q_{i}\in \mathcal{C}^{\eta}(K)$ for $i \in \mathcal{N}$ and $\{x_{1i_1}\}_{i_1=0}^{n_1}$ be a equally spaced points collection.
	If   $\gamma_0>N^{1-\eta}$ and $\underset{r\to \infty}{\lim}\frac{N(r)}{\left(N^{2-\eta}\right)^r}=\infty$, then
	$$1+\log_N (\gamma_0) \leq \underline{\dim}_B\left(G_{f^*}\right)\leq \overline{\dim}_B\left(G_{f^*}\right)\leq 1+\log_N (\gamma),$$
	where $\gamma_0:=\sum_{i\in \mathcal{N}}\|s_i\|_{0}$ and $\|s_i\|_{0}=\inf\{|s_i(x)|:x\in K\}$.
\end{theorem}

\begin{proof}
	From assumption, we have $m=1, r_i=\frac{1}{N}$ for $i\in \mathcal{N}$ and $\Lambda_0=\Lambda=N$.	From Theorem \ref{912thm1}, we get the upper bound of $\overline{\dim}_B\left(G_{f^*}\right)$.\\
	For $i\in \mathcal{N},\omega\in \mathcal{N}^k,k\in \mathbb{N}$ and $x,x'\in l_{\omega}(K)$, we have
	\begin{align*}
		&|f^*(l_{i}(x))-f^*(l_{i}(x'))|\\
		&\overset{(\ref{2101e1})}{\geq}\|s_{i}\|_{0}|f^*(x)-f^*(x')|-\|f^*\|_{\infty}|s_{i}(x)-s_{i}(x')|-|q_{i}(x)-q_{i}(x')|.
	\end{align*}	
	For $i\in \mathcal{N},\omega\in \mathcal{N}^k$ and $k\in \mathbb{N}$, we get	$G_{f^*,(i,\omega)}$ must cover at least  a rectangle of size  $\frac{|K|}{N^{k+1}}\times \frac{\|s_{i}\|_{0}({N}(k,\omega)-2)|K|}{N^k}- \|f^*\|_{\infty}	\text{Osc}_{\omega}(s_{i})-	\text{Osc}_{\omega}(q_{i})$.\\
	From (\ref{2424eq1}),  we have
	\begin{align*}
		&{N}(k+1)=\underset{i\in \mathcal{N}}{\sum}\underset{\omega\in \mathcal{N}^k}{\sum}{N}(k+1,(i,\omega))\\
		& ~~ \geq \underset{i\in \mathcal{N}}{\sum}\underset{\omega\in \mathcal{N}^k}{\sum}\frac{N^{k+1}}{|K|}\left( \frac{\|s_{i}\|_{0}({N}(k,\omega)-2)|K|}{N^k}- \|f^*\|_{\infty}	\text{Osc}_{\omega}(s_{i})-	\text{Osc}_{\omega}(q_{i})\right)\\
		& ~~ \geq N\gamma_0 ({N}(k)-2N^k)- \frac{N^{k+1}}{|K|}\underset{i\in \mathcal{N}}{\sum}\left(\|f^*\|_{\infty}	[s_{i}]_{\eta}+	[q_{i}]_{\eta}\right)N^{k\left(1-\eta\right)}\\
		& ~~ \geq N\gamma_0{N}(k)- N^{k(2-\eta)}C',
	\end{align*}
	for  $k\in \mathbb{N}$,	where $C'=\frac{N}{|K|}\underset{i\in \mathcal{N}}{\sum}\left(\|f^*\|_{\infty}	[s_{i}]_{\eta}+	[q_{i}]_{\eta}\right)+2N\gamma_0$.\\
	Via the mathematical induction method, we  get
	\begin{align*}
		&N(k)\geq  \left(N\gamma_0\right)^{k-r}N(r)-C'\left(\left(N\gamma_0\right)^{k-r-1}\left(N^{2-\eta}\right)^r+\left(N\gamma_0\right)^{k-r-2}\left(N^{2-\eta}\right)^{r+1}\right.\\
		&~~~~~~~~~\quad\left.+\cdots+\left(N^{2-\eta}\right)^{k-1}\right)\\
		&~~= \left(N\gamma_0\right)^{k-r}N(r)-C'\left(N\gamma_0\right)^{k-r-1}\left(N^{2-\eta}\right)^r\left(1+\frac{N^{1-\eta}}{\gamma_0}+\cdots+\left(\frac{N^{1-\eta}}{\gamma_0}\right)^{k-r-1}\right)\\
		&~~	\geq \left(N\gamma_0\right)^{k-r}N(r)-C'\left(N\gamma_0\right)^{k-r-1}\left(N^{2-\eta}\right)^r\frac{1}{1-\frac{N^{1-\eta}}{\gamma_0}}\\	
		&~~	= \left(N\gamma_0\right)^{k}\left(\frac{N^{2-\eta}}{N\gamma_0}\right)^r\left(\frac{N(r)}{\left(N^{2-\eta}\right)^r}-\frac{C'}{N\gamma_0\left(1-\frac{N^{1-\eta}}{\gamma_0}\right)}\right),
	\end{align*} 
	for $k>r$.\\
	By assumption, there exists $r'\in \mathbb{N}$ such that
	$$\frac{N(r')}{\left(N^{2-\eta}\right)^{r'}}-\frac{C'}{N\gamma_0\left(1-\frac{N^{1-\eta}}{\gamma_0}\right)}>0.$$
	Thus
	$$N(k)\geq \left(N\gamma_0\right)^{k}C'',$$
	for  $k>r'$, where $C''=\left(\frac{N^{1-\eta}}{\gamma_0}\right)^{r'}\left(\frac{N(r')}{\left(N^{2-\eta}\right)^{r'}}-\frac{C'}{N\gamma_0\left(1-\frac{N^{1-\eta}}{\gamma_0}\right)}\right)$.\\
	Hence
	\begin{align*}
		\underline{\dim}_B(G_{f^*})=\liminf_{k \to \infty}\frac{\log {N}(k) }{\log \left(\frac{N^{k}}{|K|}\right)}\geq \lim_{k \to \infty} \frac{\log \left(N\gamma_0\right)^{k} }{\log N^{k}}=1+\frac{\log \gamma_0}{\log N}.
	\end{align*}
\end{proof}

\begin{corollary}
	Let  $f^*$ be a  FIF on an interval,   $s_{i},q_{i}$  be continuous bounded variation maps for $i \in \mathcal{N}$,  $\{x_{1i_1}\}_{i_1=0}^{n_1}$ be a equally spaced points collection.		
	If the interpolation points are not collinear (i.e., there exists $L\neq 0$ as in Lemma \ref{1701l1}),  and  either $\gamma_{1,0}>1$ (or) $L>0$  and $\gamma_{2,0}>1$ (or) $L<0$  and $\gamma_{3,0}>1$,  
	then 
	$$1+\log_N (\gamma_0) \leq \underline{\dim}_B\left(G_{f^*}\right)\leq \overline{\dim}_B\left(G_{f^*}\right)\leq 1+\log_N (\gamma).$$
\end{corollary}
\begin{proof} From assumption, we get 
	$s_{i},q_{i}\in \mathcal{C}^{\eta}(K)$ with $\eta=1$ for $i \in \mathcal{N}$.\\
	Let us suppose that $\gamma_{1,0}>1$.\\	
	From Lemma  \ref{1701l1}, for $k\in \mathbb{N},$ we get   	
	\begin{equation}
		\nonumber	{N}(k)=\underset{\omega\in \mathcal{N}^k}{\sum}{N}(k,\omega)\geq \underset{\omega\in \mathcal{N}^k}{\sum} \frac{N^{k}}{|K|}s_{\omega_1,1,0}s_{\omega_2,1,0}\dots s_{\omega_k,1,0}|L|=\frac{N^{k}\gamma_{1,0}^k |L|}{|K|}.
	\end{equation}
	Since $\gamma_{1,0}>1$, we get
	$$\underset{k\to \infty}{\lim}\frac{N(k)}{N^k}=\infty.$$
	Since  $\gamma_0\geq \gamma_{1,0}>1$, by using Theorem \ref{912thm111}, we conclude 
	$$1+\log_N (\gamma_0) \leq \underline{\dim}_B\left(G_{f^*}\right)\leq \overline{\dim}_B\left(G_{f^*}\right)\leq 1+\log_N (\gamma).$$
	In a similar way, we can  prove the  cases.
\end{proof}
\begin{example}\label{0102ex2}
	Let us consider the data set $\left\{(x_0,0),(x_1,1/2),(x_2,1/3),(x_3,0)\right\}$ and the signature $\epsilon_1=(0,0,0)$, where $0=x_0<x_1<x_2<x_3=1$.\\
	Let  $g_{i}:[0,1]\times \mathbb{R} \to  \mathbb{R},i\in \{1,2,3\}$   given by 
	\begin{align*}
		g_{1}(x,z)=\frac{x^{\eta_1}}{2}+\frac{f(x)z}{4},
		g_{2}(x,z)=\frac{-x^{\eta_2}}{6}+\frac{z}{2}+\frac{1}{2},
		g_{3}(x,z)=\frac{-x^{\eta_3}}{3}+\frac{3z}{4}+\frac{1}{3},
	\end{align*}
	for $(x,z)\in [0,1]\times \mathbb{R}$, where $f(x)=\sin(x)$ for $x\in [0,1]$ or $f  (x)=1$ for $x\in [0,1],\eta_1\in (0,1]$ and $\eta_2,\eta_3\in [1,\infty)$.\\
	We have:\\
	$-$~ 	$s_i,q_i\in \mathcal{C}^{\eta}([0,1])$ for $i\in \{1,2,3\}$, where $\eta=\min\{1,\eta_1\}$;\vspace{1.5mm}\\ 
	$-$~$\gamma=\frac{3}{2},\gamma_{2,1}=\frac{5}{4}$ if $f\equiv\sin$ and $\gamma_{2,1}=\gamma$ if   $f\equiv1$;\vspace{1.5mm}\\
	$-$~$\gamma\leq\frac{3}{\Lambda^{\eta}}$ if $\eta\leq \log_{\Lambda}2$ and $\gamma>\frac{3}{\Lambda^{\eta}}$ if $\eta> \log_{\Lambda}2$.\vspace{1.5mm}\\
	$\mathit{Case}~(i).$ Consider $x_1=4/15$ and $x_2=3/5$, then we get
	$\Lambda_0=\frac{15}{4}$, $\Lambda=\frac{5}{2}$, $x_1=\left(1-\frac{4}{9}\right)x_0+\frac{4}{9}x_2$ and $L=p(x_1)-\left(\left(1-\frac{4}{9}\right)p(x_0)+\frac{4}{9}p(x_2)\right)>0$.\vspace{1.5mm}\\			
	Based on Theorem \ref{912thm1} and \ref{2301t1}, we get
	\begin{align*}	\underline{\dim}_B\left(G_{f^*}\right) \geq 	\begin{cases}			1+\log_{\frac{15}{4}}\left(\frac{5}{4}\right),&\text{if~} f\equiv\sin,\\			1+\log_{\frac{15}{4}}\left(\frac{3}{2}\right),&\text{if~} 	f\equiv1		\end{cases}	\end{align*}	
	and
	\begin{align*}
		\overline{\dim}_B\left(G_{f^*}\right)\leq \begin{cases}
			1-\eta+\log_{\frac{5}{2}}3, &\text{if~} 	\eta\leq \log_{2.5}2,\\
			1+\log_{\frac{5}{2}} \left(\frac{3}{2}\right), &\text{otherwise}.
		\end{cases}
	\end{align*}	
	$\mathit{Case}~(ii).$ Consider $x_1=1/3,x_2=2/3,f\equiv 1$ and $\eta_1>\log_32$, then we get
	$\Lambda_0=\Lambda=3,\gamma_{2,1}=\gamma$,~$x_1=\left(1-\frac{1}{2}\right)x_0+\frac{1}{2}x_2$ and $L=p(x_1)-\left(\left(1-\frac{1}{2}\right)p(x_0)+\frac{1}{2}p(x_2)\right)>0$.\vspace{1.5mm}\\		
	From Corollary \ref{3001cor1}, we get
	$$\dim_B(G_{f^*})=1+\frac{\log\left(1.5\right)}{\log 3}.$$
	
	\begin{figure}
		\centering
		\includegraphics[scale=.55]{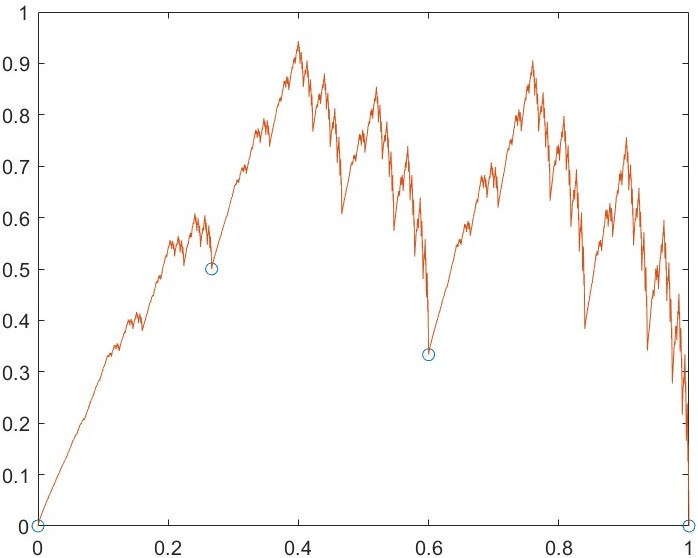}
		\caption{Graph of FIF for  $f\equiv \sin$ with $(\eta_1,\eta_2,\eta_3)=(0.8,2,1)$ derived from the\\ Case $(i)$ of Example \ref{0102ex2}  and we get $1.1688\leq \underline{\dim}_B\left(G_{f^*}\right)\leq  \overline{\dim}_B\left(G_{f^*}\right)\leq 1.44251$.}\label{fig1}
	\end{figure}
	\begin{figure}
		\centering
		\includegraphics[scale=.55]{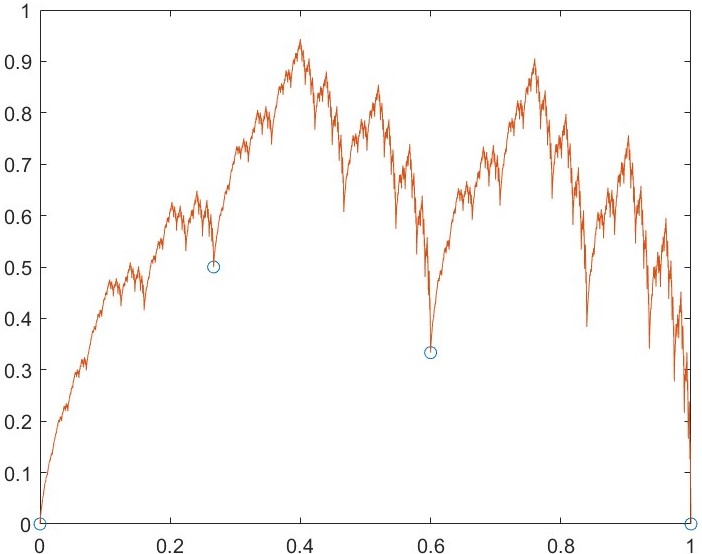}
		\caption{Graph of FIF for  $f\equiv 1$ with $(\eta_1,\eta_2,\eta_3)=(0.8,2,1)$ derived from the\\ Case $(i)$ of Example \ref{0102ex2}  and we get $1.3067\leq \underline{\dim}_B\left(G_{f^*}\right)\leq \overline{\dim}_B\left(G_{f^*}\right)\leq 1.44251$.}\label{fig1}
	\end{figure}
\end{example}

\section{Conclusion}
According to the research done so far on the dimensional analysis of FIFs, FIFs on specific domains have been taken and their box dimensions have been analyzed. It is important to note that each domain so far considered can be written as the attractor of some suitable IFS. In this paper, we have analyzed the box dimension of FIF with its domain considered as an attractor of an arbitrary IFS. 
Hence, most of the FIFs constructed so far will become particular cases of our theory. Also, studies have shown that  the domain of any new FIF can be likely be obtained  as an attractor of a suitable IFS. Hence, in the future, if any new FIF is considered in a different domain, our dimensional results can be applied to such FIFs as well. Hence, the theory on which we have worked will act as a single platform for the study of the dimensional analysis of a large class of FIFs.


\begin{thebibliography}{10}
	\bibitem{agn} Akhtar, M.N., Guru Prem Prasad, M., Navascu\'{e}s, M.A.: Box dimension of $\alpha$-fractal function with variable scaling factors in subintervals. Chaos Solitons Fractals \textbf{103}, 440--449 (2017)
	
	\bibitem{b1}    Barnsley, M.F.: Fractal functions and interpolation. Constr. Approx. \textbf{2}(2), 303--329 (1986) 
	
	\bibitem{bm} Barnsley, M.F., Massopust,  P.R.: Bilinear fractal interpolation and box dimension.  J. Approx. Theory	\textbf{192}, 362--378 (2015) 
	
	
	\bibitem{bdd}  	 Bouboulis, P., Dalla, L., Drakopoulos,  V.: Construction of recurrent bivariate fractal interpolation surfaces and computation of their box-counting dimension.  J. Approx. Theory \textbf{141}(2), 99--117 (2006)  
	
	
	\bibitem{bd}  	 Bouboulis, P.,  Dalla, L.: A general construction of fractal interpolation functions on $\mathbb{R}^n$.  European J. Appl. Math. \textbf{18}(4), 449--476 (2007)
	
	\bibitem{car}  Carvalho, A.: Box dimension, oscillation and smoothness in function spaces. J. Funct. Spaces Appl. \textbf{3}(3), 287--320 (2005) 
	
	\bibitem{cvvt}  Chand, A.K.B., Vijender, N., Viswanathan, P., Tetenov,  A.V.: Affine zipper fractal interpolation functions. Bit Numer. Math. \textbf{60}(2), 319--344 (2020) 
	
	
	\bibitem{DSO}  \c{C}elik, D., Ko\c{c}ak, \c{S}., \"{O}zdemir,   Y.:
	Fractal interpolation on the Sierpinski gasket. J. Math. Anal. Appl. \textbf{337}(1), 343--347 (2008) 
	
	\bibitem{DEL}  Deliu, A., Jawerth,  B.: Geometrical dimension versus smoothness. Constr. Approx. \textbf{8}, 211--222  (1992) 
	
	\bibitem{FAL} Falconer, K.: Fractal Geometry: Mathematical Foundations and Applications. 	John Wiley $\&$ Sons,	Chichester (1990)
	
	\bibitem{fesu}   Feng, Z., Sun, X.: Box-counting dimensions of fractal interpolation surfaces derived from fractal interpolation functions.  J. Math. Anal. Appl. \textbf{412}(1), 416--425 (2014) 
	
	\bibitem{GEHA}   Geronimo, J.S., Hardin, D.: Fractal interpolation surfaces and a related 2-D multiresolution analysis.  J. Math. Anal. Appl. \textbf{176}(2), 561--586 (1993) 	
	
	
	\bibitem{di1}  Hardin, D.P., Massopust, P.R.: The capacity	for a class of fractal functions.  Comm. Math. Phys. \textbf{105}(3), 455--460 (1986) 
	
	\bibitem{HUT} Hutchinson, J.: Fractals and self-similarity. Indiana Univ. Math. J. \textbf{30}, 713--747 (1981) 	
	
	\bibitem{jcns} Jha, S., Chand, A.K.B., Navascu\'{e}s, M.A., Sahu, A.:
	Approximation properties of bivariate $\alpha$-fractal functions and dimension results. Appl. Anal. \textbf{100}(16), 3426--3444 (2021)
	
	\bibitem{KI}   Kigami, J.: Harmonic calculus on p.c.f. self-similar sets.  Trans. Amer. Math. Soc. \textbf{335}(2), 721--755 (1993) 
	
	\bibitem{KUM}   Kumar, D.,  Chand, A.K.B., Massopust,  P.R.: Multivariate zipper fractal functions.  Numer. Funct. Anal. Optim. \textbf{44}(14), 1538--1569 (2023) 
	
	\bibitem{MAL}  Ma\l ysz, R.:  The Minkowski dimension of the bivariate fractal interpolation surfaces. Chaos Solitons Fractals \textbf{27}, 1147--1156 (2006) 
	
	\bibitem{Ma1}   Massopust, P.R.: Fractal surfaces.  J. Math. Anal. Appl. \textbf{151}(1), 275--290 (1990) 
	
	
	\bibitem{mass} Massopust, P.R.: Fractal Functions, Fractal Surface, and Wavelets.  Academic Press, New York (1994)
	
	\bibitem{nava} Navascu\'{e}s, M. A.: Approximation of fixed points and fractal functions by means of different iterative algorithms.	Chaos Solitons Fractals \textbf{180}, 1--7 (2024)
	
	\bibitem{fis-13} Pandey, K.K., Viswanathan,  P.V.: Multivariate fractal interpolation functions: some approximation aspects and an associated fractal interpolation operator. Electron. Trans. Numer. Anal. \textbf{55}, 627--651 (2022) 
	
	\bibitem{PRS} Prokaj, R.D., Raith, P., Simon, K.: Fractal dimensions of continuous piecewise linear iterated function systems. Proc. Amer. Math. Soc. \textbf{151}(11), 4703--4719 (2023)
	
	\bibitem{rua}   Ruan, H.J.: Fractal interpolation functions on post critically finite self-similar sets. Fractals \textbf{18}(1), 119--125 (2010) 	
	
	\bibitem{bdim}  Sahu, A., Priyadarshi, A.: On the box-counting dimension of graphs of harmonic functions on the Sierpi\'nski gasket. J. Math. Anal. Appl. \textbf{487}(2), 1--16 (2020)
	
	
	
\end{thebibliography}

\end{document}